\documentclass[12pt,reqno]{amsart}
\usepackage{amsmath,amssymb}
\usepackage[pdftex,colorlinks,citecolor=blue]{hyperref}
\usepackage{tikz}
\theoremstyle{definition}
\newtheorem{theorem}{Theorem}[section]
\newtheorem{definition}[theorem]{Definition}

\newtheorem{lemma}[theorem]{Lemma}
\newtheorem{proposition}[theorem]{Proposition}

\theoremstyle{remark}
\newtheorem{remark}[theorem]{Remark}
\newtheorem{example}[theorem]{Example}

\newcommand{\arxiv}[1]{\href{http://arxiv.org/abs/#1}{\textsf{arXiv:#1}}}

\newcommand{\crk}{\operatorname{crk}}
\newcommand{\ssm}{\smallsetminus}

\newcommand{\rk}{\operatorname{rk}}

\newcommand{\C}{{\mathbb{C}}}
\newcommand{\Z}{{\mathbb{Z}}}
\newcommand{\Q}{{\mathbb{Q}}}

\newcommand{\IH}{I\! H}
\newcommand{\IC}{\mathop{\mathbf{IC}}}

\newcommand{\cal}{\mathcal}

\newcommand{\cH}{{\cal H}}

\renewcommand{\setminus}{\smallsetminus}

\newcommand{\Pal}{\mathsf{Pal}}

\title{Kazhdan--Lusztig polynomials of matroids under deletion}

\author{Tom Braden}
\address{Department of Mathematics and Statistics, University of Massachusetts, Amherst, MA.}
\email{braden@math.umass.edu}

\author{Artem Vysogorets}
\address{Center for Data Science, NYU.}
\email{us441@nyu.edu}

\begin{document}

\begin{abstract}
We present a formula which relates the Kazhdan--Lusztig polynomial of a matroid $M$, as defined by Elias, Proudfoot and Wakefield, to the Kazhdan--Lusztig polynomials of the matroid  obtained by deleting an element, and various contractions and localizations of $M$.   We give a number of applications of our formula to Kazhdan--Lusztig polynomials of graphic matroids, including a simple formula for the Kazhdan--Lusztig polynomial of a parallel connection graph.
\end{abstract}

\maketitle

%%%%%%%%%%%%%%%%%%%%%%%%%%%%%%%%%%%%%%%%%%%%%%%%%%%%%%%%%%%%%%%%%%%%%%%%%%%
\section{Introduction}
%\label{sec:intro}
%%%%%%%%%%%%%%%%%%%%%%%%%%%%%%%%%%%%%%%%%%%%%%%%%%%%%%%%%%%%%%%%%%%%%%%%%%%

In \cite{EPW}, Elias, Proudfoot and Wakefield defined a polynomial
invariant $P_M(t)$ associated to any matroid $M$, which they 
called the \textbf{Kazhdan--Lusztig polynomial} of $M$.  Their definition
is formally similar to the polynomials $P_{x,y}(t)$
that were defined by Kazhdan and Lusztig \cite{KL} for elements $x, y$ in a Coxeter group $W$.  The coefficients of $P_M(t)$ depend only on the lattice of flats $L(M)$, and in fact they are integral linear combinations of the flag Whitney numbers counting chains of flats with specified ranks.

In this paper, we study how $P_M(t)$ behaves under deletion of an element from the ground set.  Our main result,
Theorem~\ref{thm:deletion formula}, is a formula relating the Kazhdan--Lusztig polynomial
of the deletion $M\ssm e$ to the Kazhdan--Lusztig polynomials of $M$ and
various contractions and localizations of $M$. Assume that $M$ is a simple matroid, and that $e$ is not a coloop of $M$.  Then our formula says that 
\begin{equation}\label{eqn:deletion formula intro}
P_M(t) = P_{M \ssm e}(t) - tP_{M_e}(t) + \sum_{F\in S} \tau(M_{F\cup e})\, t^{(\crk F)/2} P_{M^F}(t).
\end{equation}
Here the sum is taken over the set $S$ of all subsets $F$ of $E \ssm e$ such that $F$ and $F \cup e$ are both flats of $M$ 
 (any such $F$ is automatically also a flat of $M\ssm e$), and $\tau(M)$ is the coefficient of $t^{(\rk M -1)/2}$ in $P_M(t)$ if $\rk M$ is odd, and zero otherwise.
We also give a similar formula
for the closely related \textbf{$Z$-polynomial}
\[Z_M(t) = \sum_{F\in L(M)} t^{\rk F}P_{M_F}(t),\]
which was introduced in \cite{PXY}.  

Since all of the matroids appearing on the right side of \eqref{eqn:deletion formula intro} have a smaller ground set than $M$ does, it is natural to apply this formula to inductive computations of $P_M(t)$.  The challenge to carrying this out successfully is the complexity of the sum in the last term.  In the final part of the paper we present some applications of our formula to graphic matroids where the sum simplifies enough to make the formula useful.

In particular, we get a very simple formula for Kazhdan--Lusztig polynomials of \textbf{parallel connection graphs}: if $G$ is obtained by gluing graphs $H_1$ and $H_2$ at an edge $e$ common to both, and $H_1 \setminus e$, $H_2 \setminus e$ are both connected, then 
\[P_G(t) = P_{G \setminus e}(t) - tP_{H_1/e}(t)P_{H_2/e}(t).\]
Here we put $P_G(t) = P_{M_G}(t)$ when $G$ is a graph.
We use this result to give a simpler proof of a formula of Liu, Xie and Yang \cite{LXY} for the Kazhdan--Lusztig polynomials of fan graphs.

%--------------------------------------------------------------------------
\subsection{Motivation from algebraic geometry}
%--------------------------------------------------------------------------

Our results and our methods in this paper are purely combinatorial, but the motivation comes from algebraic geometry.  In this section, which is not needed for the rest of the paper, we briefly explain the geometry behind the formula \eqref{eqn:deletion formula intro}.
 
%Just as when the Coxeter group is a Weyl group a classical Kazhdan--Lusztig polynomial $P_{x,y}(t)$ is  
%the Poincar\'e polynomial (in $t^{1/2}$) of the local intersection cohomology of a Schubert variety at an appropriate point, 
The Kazhdan--Lusztig polynomial of a realizable matroid $M$ is the local intersection cohomology Poincar\'e polynomial of a variety defined as follows.  Suppose that $M$ is realized by a spanning collection $w_1, \dots, w_{n}$ of nonzero vectors in a vector space $W \cong \C^{d}$, where $d = \rk M$.  This induces a surjective 
map $\C^{n} \to W$, and dualizing gives an injection $W^* \to \C^{n}$.  Let $V\cong \C^d$ be the image of this map, and define $Y = Y(w_1,\dots, w_n)$ to be the closure of $V$ inside $(\mathbb P^1_\C)^n$.  Then
$P_M(t)$ is the Poincar\'e polynomial  of the local intersection cohomology of $Y$ at the most singular point $\infty^n$ and $Z_M(t)$ is the Poincar\'e polynomial  of the total intersection cohomology $\IH^{\bullet}(Y ;\Q)$.  (All intersection cohomology groups considered in this discussion vanish in odd degrees, and all Poincar\'e polynomials should be taken in $t^{1/2}$.)

The variety $Y$ was called the \textbf{Schubert variety} of $V$ in \cite{PXY}, because of the similarities it has with the geometry of Schubert varieties in flag varieties of reductive groups. 
In particular, $Y$ has a stratification $Y = \coprod_{F \in L(M)} C_F$ 
by affine spaces $C_F \cong \C^{\rk F}$ indexed by flats of $M$; the strata are orbits of the natural action of the additive group $(V, +)$ on $Y$.   
Closures of strata and normal slices to strata are again varieties of the same type.  (Note that for a Schubert variety in a flag variety, a normal slice to a Schubert cell cannot in general be identified with another Schubert variety.) 
The closure $Y^F := \overline{C_F}$ of a stratum is isomorphic to the variety 
$Y(w_{i_1}, \dots, w_{i_k})$, where $F= \{i_1,\dots, i_k\}$, and the vector space $W$ is replaced by the span of $w_{i_1},\dots, w_{i_k}$.
A normal slice to $Y$ at a point of $C_F$ is isomorphic to $Y(\bar w_{j_1}, \dots, \bar w_{j_r})$, where $\{j_1,\dots, j_r\} = \{1,\dots, n\} \setminus F$ and $\bar w_j$ is the image of $w_j$ in the quotient $W/\operatorname{span}(w_{i_1}, \dots, w_{i_k})$.  These varieties correspond to the localization and contraction matroids $M^F$ and $M_F$, respectively.  (See the beginning of the next section for definitions and notation of localization and contraction.)

Suppose that the element we are deleting from $M$ is $e = n$.  Then our assumption that $n$ is not a coloop means that $w_1, \dots, w_{n-1}$ still span $W$, and following the same construction shows that the variety $Y' = Y(w_1,\dots, w_{n-1})$ associated to the deletion $M\ssm e$ is the image of $Y$ under the projection 
$(\mathbb P^1)^n \to (\mathbb P^1)^{n-1}$ which forgets the last factor.  Let $p\colon Y \to Y'$ denote the map induced by this projection.
We can define a stratification $Y' = \coprod_{G \in L(M\ssm e)} C'_{G}$ the same way as before, and the map $p\colon Y \to Y'$ sends strata to strata.

The fibers of $p$ are easy to describe: either
 $p^{-1}(x)$ is a single point or it is isomorphic to $\mathbb P^1$, and it is $\mathbb P^1$ if and only if $x$ lies in a  
stratum $C'_F$ where $F$ and $F\cup e$ are flats of $M$, i.e.\ $F$ is in the set $S$ summed over in \eqref{eqn:deletion formula intro}.
%For all such flats $F$ we have 
%$p^{-1}(C'_F) = C_F \cup C_{F\cup e}$.
Because of this, the decomposition theorem of 
Beilinson, Bernstein, Deligne and Gabber takes a particularly simple form: the 
 direct image $p_*\IC(Y; \Q)$ of the intersection complex of $X$ 
% under the projection $p\colon X \to X'$ 
 is isomorphic to a direct sum
\begin{equation}\label{eqn:semi-small decomposition} \IC(Y';\Q) \oplus \bigoplus_{F\in S} \IC(\overline{C'_F};\Q)^{\oplus \tau(M_{F\cup e})}[-(\crk F)/2] .
\end{equation}
 
% a direct sum of shifted intersection complexes $\IC(\overline{C'_G}; \Q)[k]$.     Using this, one can show that for such a flat $\IC(\overline{C'_F}; \Q)[k]$ can appear as a summand of $A$ only if $\crk G$ is even and $k = -(\crk G)/2$, and furthermore the multiplicity with which it appears is \[\tau(M_{F\cup e}) = \dim \mathbb H^{(\crk F)/2 - 1}(\IC(X;\Q)|_{C_{F\cup e}}).\]

Our formula \eqref{eqn:deletion formula intro} comes from taking the stalk cohomology of $p_*\IC(Y; \Q)$ at the point stratum $C'_{\emptyset}$.  By proper base change this is
\[\mathbb H^\bullet(\IC(Y;\Q)|_{p^{-1}(C'_\emptyset)}) =\mathbb H^\bullet(\IC(Y;\Q)|_{C_\emptyset \cup C_e}),\]
which has Poincar\'e polynomial $P_M(t) + tP_{M_e}(t)$, while  the stalk of the sum \eqref{eqn:semi-small decomposition}
has Poincar\'e polynomial given by the remaining terms of \eqref{eqn:deletion formula intro}.

Our formula is analogous to the convolution formula
\begin{equation}\label{eqn:Hecke convolution}
C_sC_w = C_{sw} + \sum_{sz < z} \mu(z,w)C_z
\end{equation}
that governs Kazhdan--Lusztig basis elements $\{C_x\}_{x\in W}$ in the Hecke algebra $\mathcal H(W)$ (see \cite[equation (22)]{H}, for instance).  Here $s$ is a simple reflection and $sw > w$.  This formula arises from analyzing a map $\widetilde{X} \to X_{sw}$ which is similar to our map $Y \to Y'$.  Here $X_{sw}$ is a Schubert variety, and $\widetilde{X}$ is a $\mathbb P^1$-bundle over a smaller Schubert variety $X_w$. Again the fibers are either points or $\mathbb P^1$, and the analysis of the decomposition theorem is essentially the same.

There is one important difference, however.  In \eqref{eqn:Hecke convolution} all of the terms except $C_{sw}$ involve basis elements $C_z$  for $z \le w$, so it gives a recursive computation of $C_{sw}$.  In fact this formula was used by Kazhdan and Lusztig \cite{KL} to prove the existence of the basis elements $C_x$.  The expression corresponding to $\widetilde{X}$ is $C_sC_w$, reflecting the structure of $\widetilde{X}$ as a $\mathbb P^1$-bundle.  On the other hand, in our situation the variety $Y$ ``upstairs" is in general more complicated than $Y'$, and doesn't have a simple relation with lower-dimensional varieties of the same type.
As a result, the power of our formula in inductive computations and proofs is  more limited.

\subsubsection*{Acknowledgements}
The authors thank Jacob Matherne and Nicholas Proudfoot for helpful suggestions on a draft of this paper, and the anonymous referee for numerous corrections and improvements.

%%%%%%%%%%%%%%%%%%%%%%%%%%%%%%%%%%%%%%%%%%%%%%%%%%%%%%%%%%%%%%%%%%%%%%%%%%%
\section{The deletion formula}
%\label{sec:intro}
%%%%%%%%%%%%%%%%%%%%%%%%%%%%%%%%%%%%%%%%%%%%%%%%%%%%%%%%%%%%%%%%%%%%%%%%%%%

%--------------------------------------------------------------------------
\subsection{Matroid terminology}
%--------------------------------------------------------------------------

Let $M$ be a matroid on a ground set $E$.  One of the many equivalent ways to define a matroid is by its flats, which are subsets of $E$ satisfying
\begin{itemize}
\item $E$ is a flat,
\item if $F, G$ are flats, then $F \cap G$ is a flat, and
\item for any flat $F$, the complement $E \ssm F$ is partitioned by the sets $G \ssm F$ where $G$ runs over all flats which cover $F$.
\end{itemize}
The set of all flats ordered by inclusion is a ranked lattice which we denote $L(M)$, and we let $\rk\colon L(M) \to \Z_{\ge 0}$ be its rank function.  All of the invariants we consider depend only on $L(M)$ up to isomorphism as a ranked poset.

We will assume that $M$ is \textbf{simple}, which means that the empty set is a flat and the rank one flats are exactly all singleton sets $\{e\}$, $e\in E$.  This is not a real restriction, as any matroid has a simplification with an isomorphic lattice of flats.  To simplify notation, we omit the braces when referring to singleton flats, or when adding or deleting a single element from a flat or matroid.

Three operations on matroids will be important.  Given any flat $F \in L(M)$, the \textbf{contraction} $M_F$ is the matroid with ground set $E \ssm F$ whose lattice of flats is 
$\{G \setminus F \mid G\in L(M)\ \mbox{ and }\ G\ge F\}$.
%isomorphic to $\{G \in L(M) \mid G \ge F\}$. 
(More precisely, since this may not be a simple matroid, we can take its simplification.)

The \textbf{localization} $M^F$
is the matroid with ground set $F$ whose lattice of flats is $\{G \in L(M) \mid G \le F\}$.  We can combine contraction and localization: for $F \le G$, the matroids 
$(M_F)^{G \setminus F}$ and $(M^G)_F$ are isomorphic, and we denote them $M_F^G$.  The reader should beware that our notation is opposite to the one used in \cite{EPW}, where $M_F$ denoted the localization and $M^F$ denoted the contraction.

The third operation is \textbf{deletion}.  In this paper we will only consider deleting a single element $e \in E$.  The deletion matroid $M \ssm e$ is a matroid on the set $E \ssm e$ whose lattice of flats is 
\[\{F \ssm e \mid F \in L(M)\}.\]
Note that the localization $M^F$ can also be expressed as the iterated deletion of all elements of $E \setminus F$.  However, in our formulas the two operations play a somewhat different role, so we will keep the terminology separate.

%--------------------------------------------------------------------------
\subsection{Kazhdan--Lusztig polynomials}
%--------------------------------------------------------------------------

In this section we define the Kazhdan--Lusztig polynomials of matroids, using an alternate definition based on a result of Proudfoot, Xu and Young \cite{PXY}.

For any integer $n\ge 0$, let  $\Pal(n) \subset \Z[t, t^{-1}]$ be the set of all Laurent polynomials
such that $f(t) = t^n f(t^{-1})$.  In other words,
$\sum_{k=-N}^N a_kt^k$ lies in $\Pal(n)$ if and only if $a_k = a_{n-k}$ for all $k$.

\begin{lemma}\label{lem:palindromic}
For any  $f\in \Z[t,t^{-1}]$ and any $d \ge 0$, there exists a unique $g \in \Z[t,t^{-1}]$ with $\deg g < d/2$ so that $f + g \in \Pal(d)$.  If $f \in \Z[t]$ and $\deg f \le d$, then $g \in \Z[t]$.
\end{lemma}

\begin{theorem}[\cite{PXY}]\label{thm:KL poly def}
There is a unique family of polynomials $P_M(t) \in \Z[t]$ defined for all matroids $M$ with the following properties:
\begin{enumerate}
\item[(a)] If $\rk M =0$ then $P_M(t) = 1$.
\item[(b)] For all matroids of positive rank, the degree of $P_M(t)$ is strictly less than $(\rk M)/2$.
\item[(c)] For all matroids $M$, the polynomial
\begin{equation}\label{eqn:Z-polynomial}
Z_M(t) := \sum_{F\in L(M)} t^{\rk F}P_{M_F}(t)
\end{equation}
is in $\Pal(\rk M)$. 
\end{enumerate}
\end{theorem}
\begin{proof}
Apply the lemma to $f = \sum_{F \in L(M) \ssm \{\emptyset\}} t^{\rk F}P_{M_F}(t)$.  The summand for the flat $E$ is $t^{\rk E} = t^{\rk M}$, while the summand for a proper flat $F$ has degree 
smaller than $\rk F + (\crk F)/2 < \rk M$.  So the whole sum has degree exactly $\rk M$.
\end{proof}

\begin{remark}
Examining this proof, we see that it proves slightly more: since $f = t^{\rk M} + $ lower order terms, we must have $P_M(0) = 1$.  In particular if $\rk M \le 2$ we have $P_M(t) = 1$.

The linear coefficient is also easy to see.  Let $d= \rk(M)$.  The degree of $t^{\rk F}P_{M_F}(t)$ is at most $d - 2$ when $\crk F > 1$, so the coefficient of $t^{d-1}$ in $f$ is $|L^{d-1}(M)|$, the number of coatoms.  The coefficient of $t$ in $f$ is clearly $|L^1(M)|$, so the coefficient of $t$ in $P_M(t)$ is
\[|L^{d-1}(M)| - |L^{1}(M)|.\]
\end{remark}

\begin{remark}
The polynomials $P_M(t)$ were originally defined a different way in \cite{EPW}, using an approach closer to the definition of classical Kazhdan--Lusztig polynomials (see \cite{P}, which uses a framework of Stanley to show the parallels between these two theories and the theory of toric $g$-polynomials of polytopes).  The polynomial $Z_M(t)$ defined by
 \eqref{eqn:Z-polynomial} was defined in 
 \cite{PXY}, where it was shown to be palindromic. Lemma~\ref{lem:palindromic} implies that our definition gives the same polynomials as the original one.   
\end{remark}

The following useful result can be proved easily using either our definition of Kazhdan--Lusztig polynomials or the one from \cite{EPW}.

\begin{proposition}[\cite{EPW},Proposition 2.7]\label{prop:direct sum}
For any matroids $M$, $M'$ we have
\[P_{M\oplus M'}(t)= P_{M}(t)P_{M'}(t).\]
\end{proposition}

In particular, if $M$ is a Boolean matroid, it is a direct sum of rank $1$ matroids, so $P_M(t) = 1$.

%--------------------------------------------------------------------------
\subsection{The $\tau$-invariant}
%--------------------------------------------------------------------------

\begin{definition}
For a matroid $M$ whose rank is odd, say $\rk(M) = 2k+1$, let $\tau(M)$ be the coefficient of $t^k$ in $P_M(t)$, in other words the coefficient of highest possible degree.  If $\rk(M)$ is even, we put $\tau(M) = 0$. 
\end{definition}

The role that the invariant $\tau(M)$ plays in our results about Kazhdan--Lusztig polynomials of matroids is analogous to the role the number $\mu_{x,y}$ plays in the classical theory of Kazhdan--Lusztig polynomials of Coxeter groups. Unlike $\mu_{x,y}$, however, $\tau(M)$ seems to very rarely vanish.  The next lemma gives one important case when $\tau(M) = 0$.

\begin{lemma}\label{lem:vanishing tau}
If $M$, $M'$ are matroids of positive rank, then
\[\tau(M\oplus M') = 0.\]
\end{lemma}
\begin{proof}
The result is trivial if $\rk(M\oplus M')$ is even, so we can suppose without loss of generality
that $\rk(M) = 2k+1$ is odd and $\rk(M') = 2\ell$ is even.  Then $\deg P_M(t) \le k$ and $\deg P_{M'}(t) \le \ell -1$, so $\tau(M\oplus M')$, which is the coefficient of $t^{k+\ell}$ in 
$P_{M\oplus M'}(t) = P_M(t)P_{M'}(t)$, must vanish.
\end{proof}

%--------------------------------------------------------------------------
\subsection{Deletion formula}
%--------------------------------------------------------------------------
We are ready to state the main result of this paper.
Let $M$ be a simple matroid with ground set $E$, and take $e\in E$.  The deletion matroid $M \ssm e$ has as flats all sets $F \ssm e$, $F \in L(M)$.  

Define a set
\begin{align*}
S & := \{F \in L(M)\mid e\notin F \mbox{ and } F\cup e \in L(M)\}. \\
%S_{ev} & := \{ F \in L(M) \mid F \in S \mbox{ and } \crk F \mbox{ is even.}\}
% \\  S^+ & := \{F \cup \{e\}\mid F \in S\}.
\end{align*}

\begin{theorem}\label{thm:deletion formula}
If $e\in E$ is not a coloop in $M$, then 
\begin{equation}\label{eqn:KL deletion formula}
P_M(t) = P_{M \ssm e}(t) - tP_{M_e}(t) + \sum_{F \in S} \tau(M_{F\cup e})\, t^{(\crk F)/2} P_{M^F}(t)
\end{equation}
and
\begin{equation}\label{eqn:Z deletion formula}
Z_M(t) = Z_{M\ssm e}(t) + \sum_{F \in S} \tau(M_{F\cup e})\, t^{(\crk F)/2}Z_{M^F}(t).
\end{equation}

\end{theorem}

Note that since $\rk(F\cup e)= \rk(F)+1$ whenever $F \in S$ and $\tau(M) =0$ if the rank of $M$ is even, either sum above can be replaced by the sum over all $F \in S$ of even corank.

\begin{example}
Let us apply the theorem to the rank $d$ uniform matroid on $d+1$ elements, which we  denote $U_{1,d}$.  For each $k < d$, its flats of rank $k$ are all size $k$ subsets of $E = \{0,\dots, d\}$.
In particular, every localization $M^F$ for $F \ne E$ is Boolean, so $P_{M^F}(t) = 1$.  Deleting any element of $E$ also results in a Boolean matroid, so $P_{M\ssm e}(t)=1$.  

On the other hand, contracting an element results in a uniform matroid of smaller rank: we have 
$M_e \cong U_{1,d-1}$, and more generally 
$M_{F\cup e} \cong U_{1,d-k-1}$, where $k = |F|$.

Let $c^k_{1,d}$ denote the coefficient of $t^k$ in $P_{U_{1,d}}(t)$.  For  $0 < k < d/2$ the degree $k$ part of the formula \eqref{eqn:KL deletion formula} gives
\begin{equation}\label{eqn:deletion and U(1,d)}
c^k_{1,d} = - c^{k-1}_{1,d-1} + \binom{d}{d-2k}c^{k-1}_{1,2k-1}.
\end{equation}
A simple formula for $c^k_{1,d}$ was established 
in \cite{PWY}: we have
\begin{equation}\label{eq:KL coeff of Cn} 
c^k_{1,d} = \frac{1}{k+1}\binom{d-k-1}{k}\binom{d+1}{k} = \frac{1}{d-k}\binom{d-k}{k+1}\binom{d+1}{k}.
\end{equation}
Substituting this into \eqref{eqn:deletion and U(1,d)} and rearranging, we have
\begin{align*}
c^k_{1,d} + c^{k-1}_{1,d-1} & = \frac{1}{d-k}\left[\binom{d-k}{k+1}\binom{d+1}{k} + \binom{d-k}{k}\binom{d}{k-1}\right] \\
%&= \frac{1}{d-k}\left[\binom{d-k}{k+1}\left( \binom{d}{k-1} + \binom{d}{k}\right)+ \binom{d-k}{k}\binom{d}{k-1}\right]\\
%&= \frac{1}{d-k}\left[\left( \binom{d-k}{k+1}+\binom{d-k}{k}\right)\binom{d}{k-1} + \binom{d-k}{k+1}\binom{d}{k}\right]\\
%&=\frac{1}{d-k}\left[\binom{d-k+1}{k+1}\binom{d}{k-1} +  \binom{d-k}{k+1}\binom{d}{k}\right]\\
%&=\frac{1}{d-k}\left[\binom{d}{k+1,k-1,d-2k}+ \binom{d}{k,k+1,d-2k-1}\right] \\
%& = \frac{1}{d-k}\left[\binom{d}{k+1}\left(\binom{d-k-1}{k+1}+\binom{d-k-1}{k}\right)\right]\\
%&=\frac{1}{d-k}\binom{d}{k+1}\binom{d-k}{k+1} = \frac{1}{k}\binom{d}{k+1}\binom{d-k-1}{k-1}\\
%&= \binom{d}{d-2k}c^{k-1}_{1,2k-1}
&=\frac{(d-k-1)! \,d!}{(k+1)!(d-2k-1)!k!(d-k+1)!}+\frac{(d-k)! d!}{k!(d-2k)!(k-1)!(d-k+1)!}\\
&=
\frac{(d-k-1)!\,d!}{(d-k+1)!(d-2k)!(k+1)!k!}\left[(d+1)(d-2k) + k(k+1)\right]\\
&= \frac{(d-k-1)!\,d!}{(d-k+1)!(d-2k)!(k+1)!k!}(d-k)(d-k+1) \\
&=\frac{d!}{(d-2k)!(k+1)!k!} \\
&= \frac{1}{k}\binom{d}{d-2k}\binom{2k}{k-1}\\
&=\binom{d}{d-2k}c^{k-1}_{1,2k-1}.
\end{align*}
Thus our formula gives a new proof of the formula \eqref{eq:KL coeff of Cn}, by induction on $d$.  Similar formulas for the coefficients of $P_{U_{m,d}}(t)$ are given in \cite{GLXYZ}.  It may be possible to prove them using our result, but we have not yet been able to do so.
\end{example}

\begin{remark}
The papers \cite{PWY,GPY,GLXYZ} actually compute a richer invariant, the \textbf{equivariant} Kazhdan--Lusztig polynomial, for uniform matroids.  For a matroid with an action of a finite group $\Gamma$, the coefficients of this polynomial are (virtual) characters of $\Gamma$ rather than integers. 
Since our formula requires choosing an element to delete and thus breaks the symmetry, it cannot be refined to an equation of equivariant Kazhdan--Lusztig polynomials for the full group that acts.  However, it should be possible to upgrade it to an equivariant formula for the action of the stabilizer of the element being deleted (we thank the referee for pointing this out to us).  It is possible that the extra structure this gives would be helpful in computing $P_{U_{m,d}}(t)$ for general $m$.
\end{remark}

%--------------------------------------------------------------------------
\subsection{Perverse elements and the KL basis}
%--------------------------------------------------------------------------

Let
$\cH = \cH(M) $ be the free $\Z[t,t^{-1}]$-module with basis indexed by $L(M)$.  In other words, elements of $\cH$ are formal sums
\[\alpha = \sum_{F \in L(M)} \alpha_F \cdot F, \;\; \alpha_F \in \Z[t,t^{-1}].\]
There is an important abelian subgroup
$\cH_p \subset \cH$, defined as the set of
all $\alpha\in \cH$ so that for every flat $F\in L(M)$ we have $\alpha_F \in \Z[t]$ and 
\begin{equation}\label{eqn:Verdier condition}
\sum_{G \ge F} t^{\rk F-\rk G}\alpha_G \in \Pal(0).
\end{equation}

\begin{remark}
We will not need this in what follows, but there is another way to describe elements satisfying the condition \eqref{eqn:Verdier condition}.  They are exactly the elements fixed by an involution $\alpha \mapsto \overline{\alpha}$ of $\cH$, defined by 
\[\overline{\alpha} = \sum_F \overline{\alpha_F}\cdot \overline{F},\]
where $\overline{\alpha_F(t)} = \alpha_F(t^{-1})$ and
\[\overline{F} = \sum_{G \le F} t^{2(\rk G-\rk F)}\chi_{M^F_G}(t^2) \cdot G.\]
Here $\chi_{M}(t)$ denotes the characteristic polynomial of $M$.
\end{remark}

For any flat $F$, define 
\[\zeta^F = \sum_G \zeta^F_G \cdot G = 
\sum_{G\le F} t^{\rk F - \rk G}P_{M^F_G}(t^{-2}) \cdot G.\]
\begin{lemma}  $\zeta^F$ lies in $\cH_p$.
\end{lemma}
\begin{proof}
Since $\deg P_{M^F_G}(t^2) < \rk F - \rk G$ unless $F = G$, we get that 
\[\zeta^F \in F + \sum_{G < F} t\Z[t]\cdot G,\]
so in particular $\zeta^F_G \in \Z[t]$ for all $G$.

To see that \eqref{eqn:Verdier condition} holds, take any flat $H\le F$.  Then we have 
\begin{align*}
\sum_{G \ge H} t^{\rk H-\rk G}\zeta^F_G & = t^{\rk F - \rk H}\sum_{H \le G\le F} (t^{-2})^{\rk G - \rk H}P_{M^F_G}(t^{-2}) \\
& = t^{\rk F - \rk H}\sum_{G' \in L(M^F_H)} (t^{-2})^{\rk G'}P_{M^F_{G'}}(t^{-2})\\
& = t^{\rk F - \rk H}Z_{M^F_H}(t^{-2}),
\end{align*}
which lies in
$t^{\rk F-\rk H}\cdot \Pal(-2\rk M^F_H) = \Pal(0)$.
\end{proof}

\begin{proposition}\label{prop:perverse sum formula}
The elements $\zeta^F$, $F\in L(M)$ form a $\Z$-basis for $\cH_p$.  For any $\beta\in \cH_p$, we have
\begin{equation}\label{eqn:perverse decomposition}
\beta = \sum_F \beta_F(0) \zeta^F.
\end{equation}
\end{proposition}
\begin{proof}
Since $\zeta^F_F = 1$ and $\zeta^F_G = 0$ unless $G \le F$, the $\zeta^F$ are linearly independent. 
To show that they span, it is enough to show the
formula \eqref{eqn:perverse decomposition}.  Take any $\beta \in \cH_p$, and let
\[\alpha = \beta - \sum_F \beta_F(0) \zeta^F.\]
We show that $\alpha_F = 0$ for all $F$, by induction on $\crk F$.  If we assume $\alpha_G = 0$ for all $G>F$, then the condition 
\eqref{eqn:Verdier condition} says that $\alpha_F \in \Pal(0)$. Together with the facts that $\alpha_F \in \Z[t]$ and $\alpha_F(0) = 0$, we immediately get $\alpha_F = 0$. 
\end{proof}

%The basis elements $\zeta^F$ first appeared in \cite{EPW}, where it was conjectured that they gave positive structure constants for a multiplicative structure on $\cH$ deforming the M\"obius algebra.  This conjecture turned out not to be true, but our results provide some evidence that this basis is still useful.

%--------------------------------------------------------------------------
\subsection{Deletion and the KL basis}
%--------------------------------------------------------------------------

Let $M$ be a simple matroid and suppose $e$ is not a coloop of $M$, so that $M$ and $M \ssm e$ have the same rank.  
We have a surjective map $L(M) \to L(M\ssm e)$
sending $F$ to $F \ssm e$.  For any flat $F \in L(M)$, define its \textbf{discrepancy} to be 
\[\delta(F) = \rk_M(F)-\rk_{M\ssm e}(F\ssm e).\]

Define a homomorphism
$\Delta \colon \cH(M) \to \cH(M\ssm e)$ by letting
\[\Delta(F) = t^{-\delta(F)}(F\ssm e)\]
and extending $\Z[t,t^{-1}]$-linearly.  Our main theorem will be a consequence of the following.

\begin{proposition}\label{prop:semi-small}
We have $\Delta(\zeta^E) \in \cH_p(M\ssm e)$.
\end{proposition}
\begin{proof}
Let
\[\beta = \sum_{G\in L(M\ssm e)} \beta_G\cdot G = \Delta(\zeta^E).\]
Since $\zeta^E \in E + \sum_{F \ne E} t\Z[t]\cdot F$, $\delta(F)\in \{0,1\}$ for every $F$, and 
$\delta(E)=0$ because $e$ is not a coloop, it follows that $\beta_G\in \Z[t]$ for every $G$.

Now take a flat $H$ of $M\ssm e$, and consider the sum
\[\sum_{\substack{G \in L(M\ssm e) \\ G \ge H}} t^{\rk H - \rk G} \beta_G
= \sum_{\substack{F \in L(M)\\ F \ssm e \ge H}} t^{\rk H - \rk(F\ssm e)}t^{\delta(F)}\zeta^E_F = \sum_{\substack{F \in L(M)\\ F \ssm e \ge H}} t^{\rk H - \rk F}\zeta^E_F.\]
Applying the following lemma now shows that this sum is in $\Pal(0)$.
\end{proof}

\begin{lemma}
For any flat $H\in L(M\ssm e)$ and any $F \in L(M)$ we have $F \ssm e \ge H$ if and only if $F \ge \bar{H}$, where $\bar H$ is the closure of $H$ in $M$.  Furthermore, we have
\[\rk_M \bar{H} = \rk_{M\ssm e} H.\]
\end{lemma}
\begin{proof}
If $F \ge \bar{H}$, then $F \ssm e \ge \bar{H} \ssm e = H$.  Conversely, if
$F\ssm e \ge H$, then $F \ge \overline{F \ssm e} \ge \overline{H}$.
\end{proof}

%--------------------------------------------------------------------------
\subsection{Proof of Theorem~\ref{thm:deletion formula}, first part}
%--------------------------------------------------------------------------

  Define $\beta = \Delta(\zeta^E)$.  Then Propositions~\ref{prop:perverse sum formula} and \ref{prop:semi-small} imply that
\begin{equation}\label{eq:beta}
\beta = \sum_{F\in L(M\ssm e)} \beta_F(0)\zeta^F.
\end{equation}
We have $\beta_{E\ssm e}(0) = \zeta^E_E(0)=1$.  The only other way a summand of \eqref{eq:beta} can be nonzero is if 
$F = G \ssm e$ for some flat $G$ of $M$ where $\delta(G) = 1$, or in other words $F \cup e$ is in the set $S$ of Theorem~\ref{thm:deletion formula}.  If that happens, we have 
\[\beta_{F}(0) = 
\mbox{coefficient of $t$ in } \zeta^E_{F\cup e} = \tau(M_{F\cup e}).\]
In other words, we have
\begin{equation}\label{eqn:decomposition in cH}
\beta = \zeta^{E \ssm e} + \sum_{F\in S} \tau(M_{F\cup e})\zeta^F.
\end{equation}

Now look at the coefficient of the empty flat in \eqref{eqn:decomposition in cH}.
By definition of $\beta = \Delta(\zeta^E)$, we have 
\begin{align*}
\beta_\emptyset  & = \zeta^E_\emptyset + t^{-1}\zeta^E_{e} \\
& = t^{\rk E}P_M(t^{-2}) + t^{-1}t^{\rk(E\ssm e) - \rk e}P_{M_{e}}(t^{-2})\\
& = t^{\rk M}(P_M(t^{-2}) + t^{-2}P_{M_e}(t^{-2})).
\end{align*}
On the other hand, we have
\begin{align*}
\beta_\emptyset & = \zeta^{E\ssm e}_\emptyset + \sum_{F\in S} \tau(M_{F\cup e}) \zeta^F_\emptyset\\
& = t^{\rk(E\ssm e)}P_{M\ssm e}(t^{-2}) + 
\sum_{F\in S} \tau(M_{F\cup e})t^{\rk F}P_{M^F}(t^{-2})\\
& = t^{\rk M}\left( P_{M\ssm e}(t^{-2}) + \sum_{F\in S} t^{-\crk F}\tau(M_{F\cup e})P_{M^F}(t^{-2})\right).
\end{align*}
The first part of Theorem~\ref{thm:deletion formula} follows.

%--------------------------------------------------------------------------
\subsection{Proof of Theorem~\ref{thm:deletion formula}, second part}
%--------------------------------------------------------------------------

To prove that the second equation of Theorem~\ref{thm:deletion formula} holds, 
it will be useful to 
consider the $\Z[t,t^{-1}]$-module map
$\Phi_M \colon \cH(M) \to \Z[t,t^{-1}]$ given by
\[\Phi_M(\alpha) = \sum_{F\in L(M)} t^{-\rk F}\alpha_F.\]
Then we have
\[\Phi_{M\ssm e} \circ \Delta = \Phi_M,\]
which can be easily checked on the basis elements $F \in L(M)$.

Furthermore, for any flat $F \in L(M)$, we have
\begin{align*}
\Phi_M(\zeta^F) & = \sum_{G \le F} t^{\rk F - 2\rk G}P_{M^F_G}(t^{-2})\\
& = t^{\rk F}Z_{M^F}(t^{-2}).
\end{align*}
Now apply this to $\beta = \Delta(\zeta^E)$.  We get
\[\Phi_{M\ssm e}(\beta) = \Phi_M(\zeta^E) = t^{\rk M}Z_M(t^{-2}).\]
On the other hand, by \eqref{eqn:decomposition in cH}, we have
\[\Phi_{M\ssm e}(\beta) = t^{\rk(M\ssm e)}Z_{M\ssm e}(t^{-2}) + \sum_{F \in S}\tau(M_{F\cup e})t^{\rk(F)}Z_{M^F}(t^{-2}).\]
Putting these two equalities together and dividing by $t^{\rk M}= t^{\rk(M\ssm e)}$ gives the desired equation \eqref{eqn:Z deletion formula} with $t^{-2}$ in place of $t$.

%%%%%%%%%%%%%%%%%%%%%%%%%%%%%%%%%%%%%%%%%%%%%%%%%%%%%%%%%%%%%%%%%%%%%%%%%%%
\section{Applications to graphic matroids}
%\label{sec:intro}
%%%%%%%%%%%%%%%%%%%%%%%%%%%%%%%%%%%%%%%%%%%%%%%%%%%%%%%%%%%%%%%%%%%%%%%%%%%

A graph $G = (V, E)$ 
%with vertex set $V$ and edge set $E$ 
gives rise to a matroid $M_G$ on the ground set $E$, whose  
independent sets are subsets of $E$ containing no cycles.  The rank of a set $S \subset E$ of edges is 
\[|V| - |\mbox{connected components of the graph } (V, S)|,\]
and its closure
is
\[\overline{S} = \left\{e = \{x,y\} \in E \mid x \mbox{ and } y \mbox{ are connected by a path in } S\right\}.\]
A set $F$ of edges is a flat if $\overline{F} = F$, or equivalently, if whenever all but one edge from a cycle of $G$ lies in $F$, the remaining edge is in $F$ as well. 

For a graph $G$, we put $P_G(t) = P_{M_G}(t)$ for the Kazhdan--Lusztig polynomial of the associated matroid, and likewise we define $\tau(G) = \tau(M_G)$.  For example, the matroid of an $n$-cycle is $M_{C_n} = U_{1,n-1}$, so by \eqref{eq:KL coeff of Cn} its 
Kazhdan--Lusztig polynomial is
\[P_{C_n}(t) = \sum_{i=0}^{\lfloor (n-1)/2\rfloor} \frac{1}{i+1}\binom{n-i-2}{i}\binom{n}{i}t^i.\] 

Not surprisingly, deletion and contraction for matroids corresponds to deleting and contracting edges: we have $M_G \ssm e = M_{G \ssm e}$ and $M_G/e = (M_G)_e = M_{G/e}$.  Note, however, that contracting $e$ can result in parallel vectors in $M_G/e$, corresponding to the version of edge contraction in which multiple edges are allowed.  Since parallel vectors do not affect the lattice of flats, it is convenient to identify any multiple edges resulting from a contraction; this corresponds to taking the simplification of the matroid $M_G/e$.

%--------------------------------------------------------------------------
\subsection{Parallel connection graphs}
%--------------------------------------------------------------------------

In this section we describe a class of graphs for which our deletion formula becomes particularly simple.

\begin{definition}
We say that a graph $G$ is the \textbf{parallel connection} of subgraphs $H_1$ and $H_2$ if 
$H_1 \cup H_2 = G$ and $H_1 \cap H_2$ is a 
single edge $e$ together with its vertices.
If this holds, the edge $e$ is called the \textbf{connection edge}. 
\end{definition}
Note that these properties imply that $H_1$ and $H_2$ are vertex-induced subgraphs of $G$.

\begin{theorem}\label{thm:edge-gluing formula}
Suppose that is $G$ is the parallel connection of subgraphs $H_1$ and $H_2$ with connection edge $e$, and 
$H_1\ssm e$, $H_2 \ssm e$ are both connected.
Then
\[P_G(t) = P_{G\ssm e}(t) - tP_{H_1/e}(t)P_{H_2/e}(t).\]
\end{theorem}

\begin{proof}
Applying Theorem~\ref{thm:deletion formula} we get
\[P_G(t) + tP_{G/e}(t) = P_{G\ssm e}(t) + \sum_{F\in S} \tau({G/(F\cup e)})P_F(t).\]
The graph $G/e$ is isomorphic to the union of $H_1/e$ and $H_2/e$ joined at a vertex, so it has the same matroid as the disjoint union of $H_1/e$ and $H_2/e$, namely $M_{H_1/e}\oplus M_{H_2/e}$. 
So by Proposition~\ref{prop:direct sum} we have
$P_{G/e}(t) = P_{H_1/e}(t)P_{H_2/e}(t)$.

Thus our result will follow if we can show that
$\tau({G/(F\cup e)}) = 0$ whenever $F \in S$.  
Let $E_i$ be the set of edges of $H_i$, and set
$F_i = F \cap E_i$ for $i = 1,2$.  Then
$G/(F\cup e)$ is isomorphic to the union of $H_1/(F_1\cup e)$ and $H_2/(F_2\cup e)$
at a vertex, so unless $F_1 \cup e$ or $F_2 \cup e$ is the entire edge set of $H_1$, $H_2$ respectively, Lemma~\ref{lem:vanishing tau} implies that $\tau({G/(F\cup e)}) = 0$.  But if 
$F_i = E_i \ssm e$, then the endpoints of $e$ are already connected by edges in $F$, so $F$ is not a flat.
\end{proof}

\newcommand{\Par}{\operatorname{Par}}
\begin{remark}
If $G$ is the parallel connection of $H_1$ and $H_2$ with connection edge $e$, the matroid $M_G$ is a \textbf{parallel connection matroid} $\Par(M_{H_1}, M_{H_2})$, as defined in \cite{B71}, for instance. 
The properties used in the proof of Theorem~\ref{thm:edge-gluing formula} still hold in this more general context.  For instance, $\Par(M_1,M_2)/e = (M_1/e) \oplus (M_2/e)$ and $\Par(M_1,M_2)/d = \Par(M_1/d,M_2)$ if $d \in E(M_1) \setminus e$.  So the same proof gives the more general formula
\[P_M(t) = P_{M\ssm e}(t) - tP_{M_1/e}(t)P_{M_2/e}(t)\]
whenever $M = \Par(M_1, M_2)$ is a parallel connection matroid with connection element $e$ and $M_1\ssm e$, $M_2\ssm e$ are connected.
\end{remark}

\begin{example}
Consider a \textbf{double-cycle} graph $C_{m,n}$ obtained as the parallel connection of an $m$-cycle and an $n$-cycle.  

\begin{figure}[!ht]
\centering
\begin{tikzpicture}[thick,acteur/.style={circle,fill=black,thick,inner sep=2pt,minimum size=0.2cm}]\label{twocy}
\node (a1) at (2.408,0) [acteur]{};
\node (a2) at (1.204,-0.7)[acteur]{}; 
\node (a3) at (1.204,-2.1) [acteur]{};
\node (a4) at (2.408,-2.8) [acteur]{};
\node (a5) at (3.612,-0.7) [acteur,blue]{};
\node (a6) at (3.612,-2.1)[acteur,blue]{}; 
\node (a7) at (4.942,-0.266)[acteur]{}; 
\node (a8) at (4.942,-2.366)[acteur]{}; 
\node (a9) at (5.75,-1.316)[acteur]{}; 
\draw[black] (a1) -- (a2);
\draw[black] (a2) -- (a3);
\draw[black] (a3) -- (a4);
\draw[black] (a1) -- (a5);
\draw[blue] (a5) -- (a6) node [midway, fill=white] {$e$};;
\draw[black] (a4) -- (a6);
\draw[black] (a5) -- (a7);
\draw[black] (a6) -- (a8);
\draw[black] (a8) -- (a9);
\draw[black] (a7) -- (a9);
\draw (2.408,-1.4) node {$C_6$};
\draw (4.592,-1.4) node {$C_5$};
\end{tikzpicture}
\caption{A double-cycle graph $C_{6,5}$.}
\label{dcex}
\end{figure}
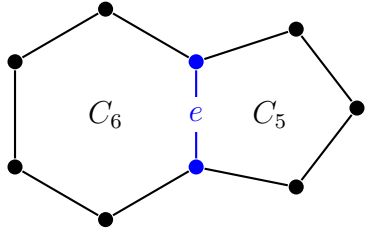

If $e$ is the connection edge, then $C_{m,n} \ssm e \cong C_{m+n-1}$.  So Theorem~\ref{thm:edge-gluing formula} gives
\begin{align*}
P_{C_{m,n}}(t) & = P_{C_{m+n-2}}(t) - tP_{C_{m-1}}(t)P_{C_{n-1}}(t),
\end{align*}
and thus the coefficient of $t^k$ in $P_{C_{m,n}}(t)$ is
\begin{align*}
& \frac{1}{k+1}\binom{m+n-k-4}{k}\binom{m+n-2}{k} \\ & - 
\sum_{i+j = k-1}\frac{1}{(i+1)(j+1)}\binom{n-i-3}{i}\binom{n-1}{i}\binom{m-j-3}{j}\binom{m-1}{j}.
\end{align*}
\end{example}

\subsection{Example: partial saw graphs}
More generally, Theorem~\ref{thm:edge-gluing formula} can be used to compute the Kazhdan--Lusztig polynomials of an iterated parallel connection of any number of cycles, or equivalently any planar graph obtained from a cycle by adding a set of non-crossing diagonals.  We illustrate this for two families of examples.  For $n \ge 3$ and $0 \le r \le n$, define a \textbf{partial saw graph}  $S_{n,r}$ to be a graph obtained by forming an iterated parallel connection with $r \le n$ three-cycles at $r$ different edges of an $n$-cycle.  Alternatively, it is an $(n+r)$-cycle with $r$ noncrossing chords added joining vertices at distance two. 
See Figure~\ref{saw}.
Note that while this can describe several different non-isomorphic graphs, all such graphs have isomorphic matroids.  We extend this to $n=2$ by letting a $2$-cycle be a single edge (or a pair of parallel edges, which has the same lattice of flats), so 
$S_{2,1} = C_3$ and $S_{2,2}$ is the parallel connection of two $3$-cycles.

\begin{figure}[!ht]
\centering
\begin{tikzpicture}[thick,acteur/.style={circle,fill=black,thick,inner sep=2pt,minimum size=0.2cm}]
\foreach \r in {0,60,...,359}
    \node at (\r:2) [acteur]{};
\foreach \r in {0,60,...,359}    
    \draw[black] (\r:2) -- (\r+60:2);
\foreach \r in {30,90,...,210}
{ 
    \node at (\r:2.5) [acteur]{};
    \draw[black] (\r-30:2)--(\r:2.5);
    \draw[black] (\r+30:2)--(\r:2.5);
};        
\end{tikzpicture}
\caption{A partial saw graph $S_{6,4}$.}\label{saw}
\end{figure}
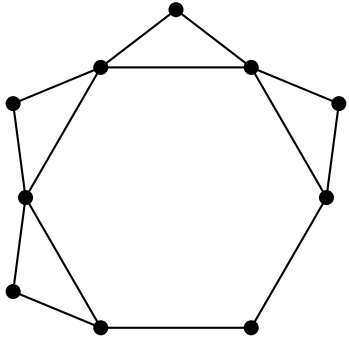

For $r > 0$, let us apply Theorem~\ref{thm:edge-gluing formula} to $S_{n,r}$, which we consider as 
the parallel connection of $S_{n,r-1}$ and $C_3$.  Let $e$ be the connection edge, so $e$ is on the central $n$-cycle and is not on any of the other $3$-cycles.
It is easy to see that $S_{n,r} \ssm e \cong S_{n+1,r-1}$ and $S_{n,r-1}/e \cong S_{n-1,r-1}$,  so our Theorem gives the following recursive formula:
\[P_{S_{n,r}}(t) =  P_{S_{n+1,r-1}}(t) - tP_{S_{n-1,r-1}}(t)P_{C_3/e}(t) = 
P_{S_{n+1,r-1}}(t) - tP_{S_{n-1,r-1}}(t),\]
valid for $n\ge 3$, $r\ge 1$.  In order to make the formula hold for $n=1,2$ we can define
$P_{S_{1,0}}(t) = P_{S_{1,1}}(t) = 0$, and $P_{S_{0,0}}(t) = t^{-1}$.

We can solve this recursion starting with $S_{n,0} = C_n$ to get the following general formula:
\begin{theorem}\label{thm:KL of Snr}
We have 
\[P_{S_{n,r}}(t) = \sum_{k=0}^r (-t)^k\binom{r}{k}p_{n+r-2k}(t),\]
where $p_m(t) = P_{C_m}(t)$ for $m \ge 2$ and 
$p_1(t) = 0$, $p_0(t) = t^{-1}$.
\end{theorem}

For example, we have
\begin{align*}
P_{S_{3,3}}(t) & = P_{C_6}(t) - t\binom31 P_{C_4}(t) + t^2\binom32 P_{C_2}(t) -t^3\cdot t^{-1} \\
& = 1+9t+5t^2 - 3t(1+2t) + 3t^2 - t^2 \\
& = 1 + 6t + t^2. 
\end{align*}

The sequence of numbers $\tau(S_{k,k})$ is the sequence of ``Motzkin sums" (\cite[sequence A00504]{OEIS}).

%--------------------------------------------------------------------------
\subsection{Fan graphs}
%--------------------------------------------------------------------------

For our second application of Theorem~\ref{thm:edge-gluing formula}, we give a simpler proof of a formula of Liu, Xie and Yang \cite{LXY} for the Kazhdan--Lusztig polynomials of fan graphs.
For $n\ge 1$, the fan graph $F_n$ is a graph with $n+1$ vertices $\{0,1,2,\dots, n\}$ and with edges 
$(0,i)$ for $1 \le i \le n$ and $(i, i+1)$ for $1 \le i \le n-1$.  Thus $F_1$ is a single edge, $F_2 \cong C_3$, and $F_3 \cong K_4 \ssm e$. 

\begin{theorem}[\cite{LXY}]\label{thm:KL of Fn}
We have 
\begin{equation}\label{eq:fan formula}
P_{F_n}(t) = \sum_{k=0}^{\lfloor \frac{n-1}2 \rfloor} \frac{1}{k+1}\binom{n-1}{k,k,n-2k-1}.
\end{equation}
\end{theorem}

In order to apply Theorem~\ref{thm:edge-gluing formula} to compute
$P_{F_n}(t)$, we need to consider a larger class of graphs.  Let $F_{n,r}$ be $F_n$ with edges 
$(0,n-r), \dots, (0,n-1)$ deleted. Thus $F_{n,0} = F_n$ and $F_{n,n-2}\cong C_{n+1}$.
For any $0 \le r \le n-3$, the graph $F_{n,r}$ is the parallel connection of  
$F_{n-r-1}$ and a copy of $C_{r+3}$ with connection edge $e = (0,n-r-1)$.
Furthermore, $F_{n-r-1}/e \cong F_{n-r-2}$ and $F_{n,r} \ssm e \cong F_{n,r+1}$, so Theorem~\ref{thm:edge-gluing formula} implies
\[P_{F_{n,r+1}}(t) - P_{F_{n,r}}(t) = tP_{C_{r+2}}(t)P_{F_{n-r-2}}(t).\]
Adding this equation for $0\le r \le n-3$, and putting $k = r+2$, we get
\begin{equation}\label{eq:relating F and C}
P_{F_n}(t) = P_{C_{n+1}}(t) - t\sum_{k = 2}^{r-1}P_{C_k}(t)P_{F_{n-k}}(t).
\end{equation}

To solve this recursion, consider the generating series
\[\Phi_C(t,u) := \sum_{n\ge 1} P_{C_{n+1}}(t)u^n, \;\;\; \Phi_F(t,u) := \sum_{n\ge 1} P_{F_n}(t)u^n.\]
Then summing $u^n$ times the equation \eqref{eq:relating F and C} gives 
\begin{equation}\label{eq:relating PhiF and PhiC}
\Phi_F(t,u) = \Phi_C(t,u) - tu\,\Phi_C(t,u)\Phi_F(t,u),
\end{equation}
so the series $\Phi_F$ and $\Phi_C$ determine each other.

In \cite{LXY} it is explained that the formula
\eqref{eq:fan formula} is equivalent to 
\begin{align*}
\Phi_F(t,u) & = \frac{2u}{1-u + \sqrt{(1-u)^2 -4tu^2}} \\
& = \frac{1}{2tu}\left[1-u - \sqrt{(1-u)^2 -4tu^2}\right].\\
\end{align*}
(Note that this formula differs from the one in \cite{LXY} because our sum for $\Phi_F(t,u)$ starts at $n=1$ instead of $n=0$.)

Plugging this into \eqref{eq:relating PhiF and PhiC}, we have
\begin{align*}
\Phi_C(t,u) & = \frac{\Phi_F(t,u)}{1- tu\Phi_F(t,u)}\\
& = \frac{\frac{1}{2tu}\left[1-u - \sqrt{(1-u)^2 -4tu^2}\right]}{1 - tu\frac{1}{2tu}\left[1-u - \sqrt{(1-u)^2 -4tu^2}\right]}\\
& = \frac{1}{tu}\cdot \frac{1-u-\sqrt{(1-u)^2 -4tu^2}}{1+u+\sqrt{(1-u)^2 -4tu^2}}\\
&= \frac{1}{tu}\cdot  \frac{1-u^2 - 2\sqrt{(1-u)^2 -4tu^2} +(1-u)^2 - 4tu^2}{(1+u)^2 - (1-u)^2+4tu^2}\\
&=\frac{1-u-2tu^2 - \sqrt{(1-u)^2 -4tu^2}}{2tu^2(1+tu)}.
\end{align*}
This agrees with the formula for $\Phi_C(t,u)$ given in \cite{PWY}, where it is also shown that this formula is equivalent to the formula \eqref{eq:KL coeff of Cn} for the coefficients of $P_{C_{n+1}}(t)$.
Thus we obtain a self-contained proof of Theorem~\ref{thm:KL of Fn} using Theorems~\ref{thm:deletion formula} and \ref{thm:edge-gluing formula}.

\begin{remark}
It is easy to see that the coefficient of $t$ in the Kazhdan--Lusztig polynomial of an $n$-cycle with $k$ non-crossing edges is 
$\binom{n}{k} - n - k$, so in particular it is independent of the edges chosen (if the diagonals are allowed to cross, however, this is no longer true).   However, Theorem~\ref{thm:KL of Fn} gives $P_{F_5}(t) = 1 + 6t + 2t^2$, and we have already seen that $P_{S_{3,3}}(t) = 1 + 6t + t^2$.  These are both triangulations of $6$-cycles, so this shows that the quadratic coefficient is sensitive to the arrangement of diagonals.  
\end{remark}

%--------------------------------------------------------------------------
\subsection{A thagomizer lemma}
%--------------------------------------------------------------------------

We finish with one more simple application of Theorem~\ref{thm:deletion formula}.  Each of our applications has relied on some simplification of the potentially complicated sum on the right side of \eqref{eqn:KL deletion formula}.  
The application to uniform matroids $U_{1,d}$ used two facts: (1) flats of a given rank are easy to count and (2) for each proper flat $F$ the localization $M^F$ is Boolean, so $P_{M^F}(t) = 1$.  On the other hand, in Theorem~\ref{thm:edge-gluing formula} all the numbers $\tau(M_{F\cup e})=0$, so all terms in the sum vanish.

Now, we give a situation in which the formula is simple because the set $S$ that is summed over is very small.  
Let $e$ be an edge of a graph $G$, and suppose that 
$G$ contains a triangle with edges $e$, $e'$, $e''$. A flat in $L(M_G)$ cannot contain exactly two of these edges of the triangle, and so a flat $F$ that is in $S$ cannot contain any edge of the triangle.

We apply this observation to the thagomizer graph $T_n$ considered in \cite{G}. This is a graph obtained from a complete bipartite graph $K_{2,n}$ by adding a single edge $e$ joining the two vertices in the first part.  Every edge of $T_n$ is part of a triangle containing $e$, and so by the previous paragraph, if we apply our deletion formula to the edge $e$, the set $S$ contains only the empty flat $\emptyset$.  Furthermore, the summand corresponding to this flat vanishes, because $G/e$ is a tree and so $\tau(G/e)=0$.  Thus we obtain the following result.

\begin{lemma}[\protect{\cite[Theorem 5.8]{GPY}}]
$P_{T_n}(t) = P_{K_{2,n}}(t) - t$.
\end{lemma}

%%%%%%%%%%%%%%%%%%%%%%%%%%%%%%%%%%%%%%%%%%%%%%%%%%%%%%%%%%%%%%%%%%%%%%%%%%%

%\bibliography{delrefs}
%\bibliographystyle{amsalpha}

\newcommand{\etalchar}[1]{$^{#1}$}
\providecommand{\bysame}{\leavevmode\hbox to3em{\hrulefill}\thinspace}
\providecommand{\MR}{\relax\ifhmode\unskip\space\fi MR }
\providecommand{\MRhref}[2]{%
  \href{http://www.ams.org/mathscinet-getitem?mr=#1}{#2}
}
\providecommand{\href}[2]{#2}

\end{document}